\documentclass[11pt,reqno]{amsart}
\usepackage{}
\usepackage{mathrsfs}
\usepackage{amsfonts}
\usepackage{a4wide}
\allowdisplaybreaks \numberwithin{equation}{section}
\usepackage{color}

\newtheorem{theorem}{Theorem}[section]
\newtheorem{proposition}[theorem]{Proposition}

\newtheorem{lemma}[theorem]{Lemma}

\newtheorem{thm}{Theorem}[section]

\newtheorem{lem}[thm]{Lemma}
\newtheorem{rem}[thm]{Remark}

\theoremstyle{definition}

\newtheorem{remark}[theorem]{Remark}

\newcommand{\R}{\mathbb{R}}
\newcommand{\ds}{\displaystyle}

\def \O{\Omega}
\def \L{{\mathcal L}}
\def \e{\varepsilon}
\def \A{{\mathcal A}}

\begin{document}
\title[simultaneous synchronized and segregated solutions]
{Infinitely many solutions with simultaneous synchronized and segregated components for nonlinear
Schr\"{o}dinger systems
 }
 \author{ Qingfang Wang and Dong Ye}

\address{School of Mathematics and Computer Science, Wuhan Polytechnic University, Wuhan 430079, P.R. China }
\email{wangqingfang@whpu.edu.cn}

\address{School of Mathematical Science, East China Normal University, Shanghai 200241, P.R. China}
\address{IECL, UMR 7502, University of Lorraine, 57050 Metz, France}
\email{dye@math.ecnu.edu.cn}
\thanks{The research was supported by NSFC (No.~12126356,12126324). D.Y.~is also supported by Science and Technology Commission of Shanghai Municipality (No.~22DZ2229014).}
\date{\today}
\begin{abstract}

In this paper, we consider the following nonlinear Schr\"odinger system in $\R^3$:
\begin{align*}
 -\Delta u_j +P_j(x) u=\mu_j u_j^3+\sum\limits_{i=1,i\neq j}^N\beta_{ij}u_i^2u_j,
\end{align*}
where $N\geq3$, $P_j$ are nonnegative radial potentials, $\mu_j>0$ and $\beta_{ij}=\beta_{ji}$ are coupling constants.
This type of systems have been widely studied in the last decade, many purely synchronized or segregated solutions are constructed, but few considerations for simultaneous synchronized and segregated positive solutions exist.
Using Lyapunov-Schmidt reduction method, we construct new type of solutions with {\bf simultaneous synchronization and segregation}. Comparing to known results in the literature, the novelties are threefold. We prove the existence of infinitely many non-radial positive and also {\sl sign-changing} vector solutions,
where some components are synchronized but segregated with other components; the energy level can be {\sl arbitrarily large}; and our approach works for {\sl any $N \geq 3$}.

\end{abstract}
\maketitle
\section{Introduction}
In this paper, we consider the following nonlinear Schr\"odinger systems:
\begin{align}\label{eq1}
 -\Delta u_j +P_j(x) u=\mu_j u_j^3+\sum\limits_{i=1,i\neq j}^N\beta_{ij}u_i^2u_j \ \mbox{ in }\;\R^3,
\end{align}
where $N\geq3$, $P_j$ are nonnegative radial potentials, $\mu_j>0$, and $\beta_{ij}=\beta_{ji}$ are coupling constants.
They arise in the study of standing waves for $N$-coupled Schr\"{o}dinger systems: $ \Phi_j=\Phi_j(x,t)\in {\mathbb C}$,
 \begin{eqnarray*}
 -i\frac{\partial}{\partial t}\Phi_j=\Delta\Phi_j-P_j\Phi_j+\mu_j|\Phi_j|^2\Phi_j+\Phi_j\sum\limits_{l=1,i\neq j}^N\beta_{ij}|\Phi_i|^2 \ \mbox{ in }\;\R^3.
\end{eqnarray*}
These type of systems, also called Gross-Pitaevskii equations, find applications in many physics problems such as
nonlinear optics and multi-species Bose-Einstein condensates (see \cite{CLLL,MS} and references therein).
For example, the system with $N=2$ arises in the Hartree-Fock theory for a double condensate,
that is, a binary mixture of a Bose-Einstein condensate in two different hyperfine states (see \cite{EGBB,E}).
More precisely, $\Phi_i$ are the wave functions of the corresponding condensates,
$\mu_j $ and $\beta_{ij} \ (j\neq i)$ are respectively the intraspecies and interspecies scattering lengths.
The signs of scattering lengths $\beta_{ij}$ determine whether the interactions of state components are repulsive or attractive.
In the attractive case, the components of a vector solution tend to go along with each other, leading to synchronization.
In the repulsive case, the components tend to segregate from each other, leading to phase separations.
These phenomena have been documented in experiments as well as in numeric simulations (see \cite{MS}).

\medskip
For the understanding of \eqref{eq1}, let us begin with the single equation, i.e~$N = 1$. It's well known \cite{Kwong} that for any $1<p<2^*$ (with $2^*$ the Sobolev critical exponent), the equation
\begin{eqnarray}\begin{cases}\label{eqs1.2}
 -\Delta w + w= w^p ,\ w>0  \mbox{ in }\;\R^d,\cr
 w(0)=\max\limits_{\R^N}w(x),  \ w \in H^1(\R^d).
 \end{cases}
\end{eqnarray}
has a unique solution denoted by $W^*$ (we omit the index $p, d$ for simplicity), hence $W^*$ is radial, and there exists $C_{d, p} >0$ such that
\begin{align}\label{wy1}
W^*(|x|) = C_{d, p}(1+O(|x|^{-1}))|x|^{\frac{1-d}{2}}e^{-|x|} \quad\hbox{as}~ |x|\rightarrow\infty.
\end{align}
If we replace the term $w$ by $Vw$ with a nonnegative potential $V$, the situation is drastically changed. In the pioneer work \cite{WY}, Wei-Yan constructed infinitely many nonradial positive solutions of the nonlinear Schr\"odinger equation
\begin{align*}
 -\Delta w + V(x)w=w^{p} \ \mbox{ in }\;\R^d, \quad w\in H^1(\R^d).
 \end{align*}
Here $1<p<2^*$, $V$ is nonnegative, continuous and radial satisfying
$$V(|x|)=1+\frac{a}{|x|^m}+O\Big(\frac{1}{|x|^{m+\sigma}}\Big), \quad \mbox{as }\; |x|\rightarrow+\infty,$$
with $a>0$, $m>1$ and $\sigma>0$.

\medskip
Later on, the two components system \eqref{eq1}, i.e.~$N = 2$ has been studied extensively in the literature. Let $\mu_1, \mu_2 > 0$, consider the following system in $\R^3$:
\begin{eqnarray}\begin{cases}\label{eqs1.3}
 -\Delta u + u=\mu_1 u^3+\beta_{12}v^2u,\cr
 -\Delta v + v=\mu_2 v^3+\beta_{12}u^2v.
 \end{cases}
 \end{eqnarray}
It's easy to see that there exist radial solutions of \eqref{eqs1.3} as follows:
\begin{align}
\label{uv}
 (U,V)=(\alpha W^*,\gamma W^*)
\end{align}
provided $-\sqrt{\mu_1\mu_2}<\beta_{12}<\min\{\mu_1,\mu_2\}$ or $\beta_{12}>\max\{\mu_1,\mu_2\}$, and
\begin{align}
\label{ag}
 \alpha=\sqrt{\frac{\mu_2-\beta_{12}}{\mu_1\mu_2-\beta_{12}^2}} ,\ \ \gamma=\sqrt{\frac{\mu_1-\beta_{12}}{\mu_1\mu_2-\beta_{12}^2}}.
\end{align}
If radial potentials $P_j$ are involved, Peng-Wang proved that under suitable conditions on $P_j$, the Schr\"odinger system
\begin{eqnarray}\begin{cases}\label{eqs1.3bis}
 -\Delta u + P_1u=\mu_1 u^3+\beta_{12}v^2u,\cr
 -\Delta v + P_2v=\mu_2 v^3+\beta_{12}u^2v,
 \end{cases}
 \end{eqnarray}
has infinitely many non-radial positive solutions of segregated type or synchronized type in $\R^3$. A key ingredient of their study is the nondegeneracy of solutions $(U, V)$ given by \eqref{uv}. More precisely, they showed that there exists a sequence $(\gamma_i) \subset(-\sqrt{\mu_1\mu_2},0)$ satisfying $\lim_{i\to\infty}\gamma_i= -\sqrt{\mu_1\mu_2}$ such that for any
\begin{align}
\label{lambda}
\beta_{12} \in \Lambda := \Big[(-\sqrt{\mu_1\mu_2},0)\backslash\{\gamma_i\}\Big] \cup(0,\min\{\mu_1,\mu_2\})\cup(\max\{\mu_1,\mu_2\},\infty),
\end{align}
$(U,V)$ is non-degenerate for the system \eqref{eqs1.3}, in the sense that the kernel in $H^1(\R^3)^2$ of the linearized system to \eqref{eqs1.3} at $(U, V)$ is given by
$${\rm Span}\Big\{\Big(\frac{\alpha}{\gamma}\frac{\partial W^*}{\partial x_j},\frac{\partial W^*}{\partial x_j}\Big),\; j=1,2,3 \Big\}$$ with $(\alpha, \gamma)$ in \eqref{ag}.
Later on, many works are realized for the two component systems, to name a few, we refer the readers to \cite{AC,BWW,EC,DWW,LW1,LW2,LP,NTTV,WZ,WW,WW1} and the references therein.

\medskip
For the constant potential case, Lin-Wei \cite{LW} established some general results for the existence issue of ground state solutions of \eqref{eq1}. In particular,
for $N=3$, they showed that the solutions are clearly of segregated type as each of the three components has a bump
moving away from each other if all $\beta_{ij}$ are negative; or one of the $\beta_{ij}$ is negative and the coefficient matrix $(\beta_{ij})$ is definitely positive. Recently,
Wei-Wu \cite{WW2} give a systematic and almost complete study on the existence of ground states to \eqref{eq1} with constant potentials and mixed couplings.
For more general potentials, Li-Wei-Wu obtain in \cite{LWW} almost optimal existence results of infinitely many non-radial
positive solutions of \eqref{eq1}, in particular, they extend the results in \cite{PV,PW} under some mild assumptions on the potentials $P_j$.

\medskip
We can remark that almost all the works in the literature showed the existence of solutions to \eqref{eq1} with purely synchronized or segregated components. To our best knowledge, the only example of simultaneous synchronized and segregated components to \eqref{eq1} was given in \cite{PWW} with $3\leq N\leq6$. For example, Peng-Wang-Wang considered the three components system where the third component has a bump at the origin, while two other components placed synchronized peaks on the vertices
of a scaled regular polygon far away from the origin, which
shows the existence of mixed phenomenon for segregation and synchronization.

\medskip
Our aim here is to show the existence of new type solutions with simultaneous synchronization and segregation for \eqref{eq1}. The novelties are {\bf threefold}. As mentioned, for $3 \leq N\leq 6$, the existence of simultaneous synchronized and segregated positive solutions with fixed peak number was showed. Here the number of spikes can be arbitrarily large; moreover we construct sign-changing solutions; and our approach works for general $N \geq 3$.


\medskip
To fix the idea, we assume first $N = 3$, see Remark \ref{remd} for more general cases. Under suitable conditions, we will construct
infinitely many non-radial positive or sign-changing solutions where two first components are synchronized but segregated from the third one. 

\medskip
The following are some technical assumptions on the potentials $(P_j)$:
\begin{enumerate}
\renewcommand{\labelenumi}{$\bullet$}
\item $P_j$ are bounded, nonnegative, radial, and there are constants $a_j\in\R,\ m_j>1$, $\sigma>0$ such that as $r\rightarrow+\infty$,
$$P_j (r)=1+\frac{a_j}{r^{m_j}}+O\Big(\frac{1}{r^{{m_j}+\sigma}}\Big), \;\; j =1,2,3. \eqno{(P)}$$
\item We say that $(P_j)$ satisfies $(H_m)$, if one of the following conditions holds true.
\begin{itemize}
\item[$(i)$] $m_3=\min\{m_1,m_2\}$, $a_3>0$; $m_1< m_2, a_1>0$, or $m_1>m_2, a_2>0$.
\item[$(ii)$] $m_1=m_2=m_3,a_1\alpha^2+a_2\gamma^2>0$ with $(\alpha, \gamma)$ given by \eqref{ag}, $a_3>0$.
\end{itemize}
\item We say that $(P_j)$ satisfies $(\widetilde{H}_m)$ if one of the following conditions holds true.
\begin{itemize}
\item[$(iii)$] $m_3=\min\{m_1,m_2\}$, $a_3<0$; $m_1< m_2,a_1<0$, or $m_1> m_2, a_2<0$.
\item[$(iv)$]$m_1=m_2=m_3$, $a_1\alpha^2+a_2\gamma^2<0$  with $(\alpha, \gamma)$ given by \eqref{ag}, $a_3<0$.
\end{itemize}
\end{enumerate}

\medskip
Our main results can be stated as follows.
\begin{theorem}\label{th1.1}
Assume that $(P)$ and $(H_m)$ are satisfied. Given any $\beta_{12} \in  \Lambda$ in \eqref{lambda}, there exists $\delta>0$ such that for any $|\beta_{13}|+|\beta_{23}|<\delta$, the system \eqref{eq1} has a sequence of non-radial positive solutions $(u_{j, \ell})$ where $u_{1, \ell}$ and $u_{2, \ell} $ are synchronized but segregated with $u_{3, \ell}$.

\smallskip
Moreover, the associated energy (see \eqref{2.1}) tends to infinity as $\ell\to \infty$;
$$\lim_{\ell\to \infty}\|u_{1, \ell}\|_\infty = \alpha\|W^*\|_\infty, \quad \lim_{\ell \to \infty} \|u_{2, \ell}\|_\infty=\gamma\|W^*\|_\infty,  \quad \lim_{\ell\to \infty} \|u_{3, \ell}\|_\infty= \|W^*\|_\infty;$$ and
\begin{align*}
\lim_{\ell\to \infty}\|\sqrt{|\mu_1-\beta_{12}|}u_{1, \ell}-\sqrt{|\mu_2-\beta_{12}|}u_{2, \ell}\|_{H^1\cap L^\infty(\R^3)} = 0.
\end{align*}
\end{theorem}
Here $\|\cdot\|_{H^1\cap L^\infty} = \|\cdot\|_{H^1} + \|\cdot\|_\infty$. Next, we introduce some notations and formulate a version which gives more precise descriptions about the segregated and synchronized character of the constructed solutions. The functional space we use is
\begin{align}\label{eqs1.4}
\begin{split}
H_{V, s} &=\Big\{u:u\in H_V^1(\R^3), u \text{\ is\  even\  in} \ x_2, x_3;\cr
 & \quad\quad u(r\cos\theta,r\sin\theta,x_3)=u\Big(r\cos\big(\theta+\frac{2\pi}{\ell}\big),r\sin\big(\theta+\frac{2\pi}{\ell}\big),x_3\Big), \;\forall \;\theta \in \R\Big\}.
 \end{split}
\end{align}
Here and after, for any function $P(x)\geq 0$, $H_{P}^1(\R^3)$ means the weighted Sobolev space, endowed with the norm
$$
\|u\|_P=\Big(\int_{\R^3}|\nabla u|^2+P(x)u^2dx\Big)^{\frac{1}{2}},
$$
induced by the inner product
$$
\langle u,v\rangle_{P}=\int_{\R^3}\Big(\nabla u\nabla v +P(x)uv\Big)dx.
$$
\begin{rem}
By the fact that $P_j$ are bounded, nonnegative and satisfy $(P)$, it's clear that $H_{P_j}^1(\R^3) = H^1(\R^3)$, and $\|\cdot\|_{P_j}$ are equivalent to the standard $H^1(\R^3)$ norm.
\end{rem}

Define $H := H_{P_1, s}^1(\R^3)\times H_{P_2, s}^1(\R^3)\times H_{P_3, s}^1(\R^3)$, endowed with the norm
$$
\|(u,v,w)\|^2 = \|u\|^2_{P_1}+\|v\|^2_{P_2}+\|w\|^2_{P_3}.
$$
Let
$$
S^k= \left(r\cos\frac{2(k-1)\pi}{\ell},r\sin\frac{2(k-1)\pi}{\ell},0\right):=(x'^k,0),\quad k=1,2,\cdots,\ell,
$$
with $r\in[r_0 \ell\ln\ell,r_1 \ell\ln\ell]$ for some $r_1>r_0>0$. Let
$$ T^k= \left(\rho\cos\frac{(2k-1)\pi}{\ell},\rho\sin\frac{(2k-1)\pi}{\ell},0\right):=(y'^k,0),\quad k=1,2,\cdots,\ell,$$
with $\rho\in[\rho_0 \ell\ln\ell, \rho_1 \ell\ln\ell]$ for some $\rho_1 > \rho_0>0$.

\medskip
To simplify, from now on, we will use $(u, v, w)$ to denote the three components $u_j$. Fix $(U, V)$ given by \eqref{uv} and $W = W^*/\sqrt{\mu_3}$. Hence $(U, V)$ satisfies \eqref{eqs1.3} and $W$ is a solution to $-\Delta w + w = \mu_3 w^3$ in $\R^3$. Our ansatz is given by
$$
U_r= \sum\limits_{k=1}^\ell U_{S^k}(x),\quad V_r= \sum\limits_{k=1}^\ell V_{S^k}(x),\quad W_\rho(x)=\sum\limits_{k=1}^\ell W_{T^k}(x)
$$
where $f_\xi(x) := \tau_\xi f(x) = f(x-\xi)$ for any $\xi\in\R^3$ and function $f$. Clearly $(U_r,V_r,W_\rho) \in H$. We will prove Theorem \ref{th1.1} by showing
\begin{theorem} \label{th1.3}
Under the assumptions of Theorem \ref{th1.1}, there are $r_1 > r_0 > 0$, $\rho_1 > \rho_0 > 0$, $\ell_0>0$ such that for any integer $\ell \geq \ell_0$, \eqref{eq1} has a solution of the form
$$
(u_\ell,v_\ell,w_\ell)=(U_{r}+\varphi,V_{r} +\psi,W_{\rho} +\xi)
$$
where $(\varphi,\psi,\xi)\in H, r \in[r_0 \ell\ln\ell,r_1 \ell\ln\ell]$, $\rho \in[\rho_0 \ell\ln\ell, \rho_1 \ell\ln\ell]$ and $$\lim\limits_{\ell\to \infty}\|(\varphi,\psi,\xi)\|_{H^1\cap L^\infty(\R^3)} = 0.$$
\end{theorem}

Roughly speaking, synchronized components are small perturbations of $(U_r,V_r)$, with sums of translated $(U,V)$ to the vertices of a large sized regular polygon; the segregated component is a small perturbation of $W_\rho$ with peaks also  localized at the vertices of a large sized regular polygon, however with a $\frac{\pi}{\ell}$ rotation shift comparing to the two first synchronized components.

\begin{remark}
\label{remd}
Our method works also for general systems with $N \geq 4$. The main idea is to organize by groups of two or one components, where the two components in the same group are synchronized, while different groups are segregated from each other. Look just the case $N = 4$. We can choose the first two components $(u_1,u_2)$ to be synchronized as above, where $u_1=U_r+\varphi_1$, $u_2=V_r + \varphi_2$, we choose the third and fourth components as follows:
$$u_3 = \sum\limits_{k=1}^\ell \widehat U(x-T^j)+\varphi_3, \quad u_4 = \sum\limits_{k=1}^\ell \widehat V(x-T^j)+\varphi_4,$$ where $(\widehat U, \widehat V)$ is a positive radial solution of a similar system to \eqref{eqs1.3}, with $\mu_3, \mu_4, \beta_{34}$ instead of $\mu_1, \mu_2, \beta_{12}$. We can prove that given $\beta_{12}, \beta_{34}\in \Lambda$, if other coupling coefficients $\beta_{ij}$ are close enough to $0$ and $P_j$ satisfy $(P)$, then a sequence of solutions to \eqref{eq1} exists such that $(u_3, u_4)$ are also synchronized between them, but segregated with $(u_1, u_2)$.
\end{remark}

For the existence of sign-changing solutions with simultaneous synchronized and segregated components, we have
\begin{theorem}\label{th1.2}
Assume that $(P)$ and $(\widetilde H_m)$ are satisfied. Given any $\beta_{12} \in \Lambda$ in \eqref{lambda}, there exists $\delta>0$ small such that for any $|\beta_{13}|+|\beta_{23}|<\delta$, \eqref{eq1} has a sequence of  non-radial solutions $(\overline u_{j, \ell})$ where $\overline u_{i, \ell}$ are all sign-changing; $\overline u_{1, \ell}$, $\overline u_{2, \ell}$  are synchronized but segregated with $\overline u_{3, \ell}$. Moreover, the associated energy tends to infinity as $\ell\to \infty$,
$$\lim_{\ell\to \infty}\|\overline u_{1, \ell}\|_\infty = \alpha\|W^*\|_\infty, \quad \lim_{\ell \to +\infty} \|\overline u_{2, \ell}\|_\infty=\gamma\|W^*\|_\infty,  \quad \lim_{\ell\to \infty} \|\overline u_{3, \ell}\|_\infty= \|W^*\|_\infty;$$
and
\begin{align*}
\lim_{\ell\to \infty}\|\sqrt{|\mu_1-\beta_{12}|}\overline u_{1, \ell}-\sqrt{|\mu_2-\beta_{12}|} \overline u_{2, \ell}\|_{H^1\cap L^\infty(\R^3)} = 0.
\end{align*}
\end{theorem}

To prove Theorem \ref{th1.2}, we denote
\begin{align}\label{wy1.4}
\begin{split}
\bar{H}_{V, s} &=\Big\{u:u\in H_V^1(\R^3), u \text{\ is\  even\  in} \ x_2, x_3;\cr
 & \quad\quad u(r\cos\theta,r\sin\theta,x_3)=-u\Big(r\cos\big(\theta+\frac{2\pi}{\ell}\big),r\sin\big(\theta+\frac{2\pi}{\ell}\big),x_3\Big), \;\forall \;\theta \in \R\Big\}.
 \end{split}
\end{align}
and $\bar{H}=\bar{H}_{P_1,s}\times\bar{H}_{P_2,s}\times \bar{H}_{P_3,s}$ endowed with the norm
$\|(u,v,w)\|$ defined as above. We will use the ansatz
$$\overline{U}_r = \sum\limits_{i=1}^{2\ell}(-1)^{i}U_{S^i}, \quad \overline{V}_r=\sum\limits_{i=1}^{2\ell}(-1)^{i}V_{S^i}, \quad \overline{W}_\rho=\sum\limits_{i=1}^{2\ell}(-1)^{i}W_{T^i}$$
and claim
\begin{theorem} \label{th1.4}
Under the assumptions of  Theorem \ref{th1.2}, there are $r_1 > r_0 > 0$, $\rho_1 > \rho_0 > 0$ and $\ell_0>0$ such that for any integer $\ell\geq \ell_0$, \eqref{eq1} has a solution of the form
$$
(\overline{u}_\ell,\overline{v}_\ell,\overline{w}_\ell) = (\overline{U}_{r} +\overline{\varphi},\overline{V}_{r} +\overline{\psi}, \overline{W}_{\rho}+\overline{\xi})
$$
where $(\overline{\varphi},\overline{\psi},\overline{\xi})\in \bar{H}$, $r\in[r_0 \ell\ln\ell,r_1 \ell\ln\ell]$, $\rho \in[\rho_0 \ell\ln\ell, \rho_1 \ell\ln\ell]$ and $$\lim\limits_{\ell\to \infty}\|(\overline{\varphi},\overline{\psi},\overline{\xi})\|_{H^1\cap L^\infty(\R^3)} = 0.$$
\end{theorem}



We will apply the Lyapunov-Schmidt reduction techniques to handle the perturbed elliptic problems.
In particular, we are inspired by \cite{PW,WY} by using the number of peaks as
a parameter in the construction of spike solutions for system \eqref{eq1}.
However we encounter some new difficulties due to the complex nonlinear couplings.

\medskip
This paper is organized as follows. In Section 2, we will introduce some preliminaries and basic estimates. In section 3, we give the proof of a key lemma for the invertibility of involved linearized operator. The energy expansion is given in section 4. We will carry out the reduction to a finite dimensional setting then prove Theorem \ref{th1.3} and \ref{th1.4} in section 5.

\section{Preliminaries and basic estimates}
Let $N = 3$ and $I$ be the energy associated to the system \eqref{eq1}, that is,
\begin{align}\label{2.1}
\begin{split}
I(u,v,w)
&=\frac{\|(u, v, w)\|^2}{2} -\frac{1}{4} \int_{\R^3}\Big(\mu_1 u^4+\mu_2 v^4+\mu_3 w^4\Big)dx\\
&\quad -\frac{1}{2}\ds\int_{\R^3} \Big(\beta_{12}u^2v^2+\beta_{13}u^2w^2+\beta_{23}v^2w^2\Big)dx.
\end{split}
\end{align}
Let $(U_{S^k}, V_{S^k}, W_{T^k})$ be defined as previously, denote
$$
X_k =\frac{\partial U_{S^k}}{\partial r}, \quad Y_k=\frac{\partial V_{S^k}}{\partial r},\quad  Z_k=\frac{\partial W_{T^k}}{\partial \rho}, \quad k = 1,\cdots,\ell.
$$
We define the space for perturbation terms $(\varphi, \psi, \xi)$ be the following.
\begin{align}
\label{spaceE}
\begin{split}
E& = \Bigl\{(\varphi, \psi, \xi) \in H,  \ds\sum\limits_{k=1}^\ell \ds\int_{\R^3}W_{S^k}^{*2}\Big(X_k \varphi +Y_k \psi\Big) dx=0, \ds\int_{\R^3}W_{T^k}^{*2}Z_k\xi dx=0, \;  k = 1,... \ell \Bigr\}\cr
& = \Bigl\{(\varphi, \psi, \xi) \in H,   \ds\int_{\R^3}W_{S^1}^{*2}\Big(X_1 \varphi +Y_1 \psi\Big) dx=0, \ds\int_{\R^3}W_{T^1}^{*2}Z_1\xi dx=0 \Bigr\}.
\end{split}
 \end{align}
The second line is ensured by the symmetry of $H$.

\medskip
For $(\varphi, \psi, \xi) \in E$, let $$J(\varphi,\psi,\xi) = I(U_r+\varphi,V_r+\psi,W_\rho+\xi),$$ we will decompose $J(\varphi,\psi,\xi)$. Recalling that $(U, V)$ resolves \eqref{eqs1.3} and $W$ satisfies $-\Delta W+ W = \mu_3 W^3$, there holds
$$
J(\varphi,\psi,\xi)=J(0,0,0)+ L_1(\varphi,\psi,\xi)+\frac{1}{2}L_2(\varphi,\psi,\xi) + R(\varphi,\psi,\xi)
$$
where
\begin{align*}
L_1(\varphi,\psi,\xi)
&= \int_{\R^3}\mu_1\Bigl(\sum\limits_{k=1}^\ell U_{S^k}^3-U_r^3\Bigr)\varphi+ \mu_2\Bigl(\sum\limits_{k=1}^\ell V_{S^k}^3-V_r^3\Bigr)\psi + \mu_3\Big(\sum\limits_{k=1}^\ell W_{T^k}^3-W_\rho^3\Big)\xi dx
\\
&\quad
+\int_{\R^3}\Big[(P_1(x) -1)U_r\varphi+(P_2(x) -1)V_r\psi+(P_3(x) -1)W_\rho\xi \Big] dx\\
&\quad
+ \beta_{12}\int_{\R^3}\Bigl(\sum\limits_{k=1}^\ell V_{S^k}^2U_{S^k}-V_r^2U_r\Bigr)\varphi dx + \beta_{12} \int_{\R^3}\Bigl(\sum\limits_{k=1}^\ell U_{S^k}^2V_{S^k}-U_r^2V_r\Bigr)\psi dx\\
& \quad - \beta_{13}\int_{\R^3}\Bigl(U_r W_\rho^2\varphi+U_r^2W_\rho\xi\Bigr)dx-\beta_{23}\int_{\R^3}\Bigl(V_r W_\rho^2\psi+V_r^2W_\rho\xi\Bigr)dx
\end{align*}
is the linear part; the quadratic part is given by
\begin{align*}
L_2(\varphi,\psi,\xi)
&= \|(\varphi, \psi, \xi)\|^2 - 3 \int_{\R^3} \mu_1 U_r^2\varphi^2 dx - 3\int_{\R^3}\mu_2 V_r^2\psi^2 dx - 3 \int_{\R^3} \mu_3 W_\rho^2\xi^2 dx\\
&\quad -\beta_{12} \int_{\R^3}\Bigl(U_r^2\psi^2+4U_r V_r\varphi\psi+V_r^2\varphi^2\Bigr)dx\\
&\quad -\beta_{13} \int_{\R^3}\Bigl(U_r^2\xi^2+4U_r W_\rho\varphi\xi+W_\rho^2\varphi^2\Bigr)dx\\
&\quad -\beta_{23} \int_{\R^3}\Bigl(V_r^2\xi^2+4V_r W_\rho\xi\psi+W_\rho^2\psi^2\Bigr)dx,
\end{align*}
and $R(\varphi,\psi,\xi)$ contains all higher order terms, that is
\begin{align}\label{wqf2.5}
\begin{split}
R(\varphi,\psi,\xi)
&=
-\int_{\R^3}\Bigl(\mu_1 U_r\varphi^3+\mu_2 V_r\psi^3+\mu_3 W_\rho\xi^3+\frac{\mu_1}{4}\varphi^4+\frac{\mu_2}{4}\psi^4+\frac{\mu_3}{4}\xi^4\Bigr)dx\\
& \quad -\beta_{12}{\mathcal R}(U_r, V_r, \varphi, \psi) - \beta_{13}{\mathcal R}(U_r, W_\rho, \varphi, \xi) - \beta_{23}{\mathcal R}(V_r, W_\rho, \psi, \xi)
\end{split}
\end{align}
where
\begin{align*}
{\mathcal R}(F, G, \eta, \zeta) := \int_{\R^3} \Big(F\eta\zeta^2 + G\zeta\eta^2 + \frac{\eta^2\zeta^2}{2}\Big) dx
\end{align*}
From now on, we will always assume
\begin{align}
\label{defr}
(r,\rho)\in D_\ell := \Big[\Big(\frac{m_3}{2\pi}-\delta \Big)\ell\ln\ell,\ M\ell\ln\ell\Big] \times \Big[\Big(\frac{m_3}{2\pi}-\delta \Big)\ell\ln\ell,\ M\ell\ln\ell\Big],
\end{align}
where $\delta>0$ is a small constant and $M > 0$ is a large constant to be chosen later, depending on $a_1,a_2$ and $A_j$ in Proposition \ref{A.1} below.

\medskip
The basic estimate for the functional $R$ is the following.
\begin{lemma}\label{lm2.4}
There exists a constant $C>0$, independent of $\ell$ large enough, such that
$$
\left\|R^{(i)}(\varphi,\psi,\xi)\right\| \leq C\|(\varphi,\psi,\xi)\|^{3-i}, \quad \forall\; \|(\varphi,\psi,\xi)\| \leq 1, \; i = 0, 1, 2.
$$
\end{lemma}
\begin{proof}
By the uniform boundedness of $U_r, V_r$ and $W_\rho$ for $\ell$ large, and the Sobolev embedding, there holds
\begin{align*}
|R(\varphi,\psi,\xi)|\leq & C\Big(\|(\varphi,\psi,\xi)\|^3+\|(\varphi,\psi,\xi)\|^4\Big).
\end{align*}
Similarly, $\|R'(\varphi,\psi,\xi)\|\leq C\|(\varphi,\psi,\xi)\|^2$, $\|R''(\varphi,\psi,\xi)\|\leq C\|(\varphi,\psi,\xi)\|$ for $\|(\varphi,\psi,\xi)\| \leq 1$. So we are done.
\end{proof}

Next we estimate the linear operator $L_1$.
\begin{lemma}\label{lm2.5}
There exists a constant $C>0$, independent of $\ell$ large, such that
$$
\|L_1\|\leq C\Big[\frac{\ell}{r^{\min\{m_1,m_2\}}} + \frac{\ell}{\rho^{m_3}}
+ (|\beta_{13}|+|\beta_{23}|)e^{-|\rho-re^{\frac{i\pi}{\ell}}|}\frac{\ell\sqrt{\ell}}{r}\Big].
$$
\end{lemma}
\begin{proof}
By the symmetry of our setting, the assumption $(P)$, the location of $r$ and the exponential decay of $U$, we get, for large enough $\ell$
\begin{align}\label{eqs3.5}
\begin{split}
\left|\int_{\R^3}(P_1-1)U_r\varphi dx \right|
& = \left|\int_{\R^3}(P_1-1)\sum\limits_{k=1}^\ell U_{S^k}\varphi \right|= \ell \left|\int_{\R^3}(P_1-1) U_{S^1}\varphi dx\right| \\
& = \ell \left|\int_{\R^3}\big[P_1(x + S^1)-1\big] U\varphi(x+ S^1)dx\right|\\
&\leq C\frac{\ell}{r^{m_1}}\|\varphi\|_{L^2}.
\end{split}
\end{align}
Similarly,
\begin{align*}
\left|\int_{\R^3}(P_2 - 1)V_r\psi dx\right| \leq C\frac{\ell}{r^{m_2}}\|\psi\|_{L^2},\quad \left|\int_{\R^3}(P_3-1)W_\rho\xi dx\right| \leq C\frac{\ell}{\rho^{m_3}}\|\xi\|_{L^2}.
\end{align*}

Next, we estimate
$$
\mu_1\int_{\R^3}\Bigl(\sum\limits_{k=1}^\ell U_{S^k}^3-U_r^3\Bigr)\varphi dx +\mu_2\int_{\R^3}\Bigl(\sum\limits_{k=1}^\ell V_{S^k}^3-V_r^3\Bigr)\psi dx.
$$
Thanks to the symmetry of $H$, we will use
\begin{align}
\label{Omegak}
\Omega_k=\Big\{z=(z' ,z_3)\in\R^2\times\R:\Big\langle\frac{z' }{|z' |},\frac{x'^{k}}{|x'^k|}\Big\rangle\geq\cos\frac{\pi}{\ell}\Big\},\quad k=1,2,\cdots\ell.
\end{align}
 Therefore, for $\ell$ large,
\begin{align}\label{eqs3.6}
\begin{split}
&
\quad \Bigl|\int_{\R^3} \Big(\sum\limits_{k=1}^\ell U_{S^k}^3-U_r^3\Big)\varphi dx\Bigr|= \ell \Bigl|\int_{\Omega_1} \Big(\sum\limits_{k=1}^\ell  U_{S^k}^3 - U_r^3\Big)\varphi dx\Bigr|\\
&\leq C \ell \int_{\Omega_1}\Bigl(U_{S^1}^2\sum\limits_{k=2}^\ell U_{S^k}+ U_{S^1}\sum\limits_{k=2}^\ell  U_{S^k}^2\Bigr)|\varphi|dx\\
&\leq C\ell \sum\limits_{k=2}^\ell \frac{e^{-|S^1-S^k|}}{|S^1-S^k|}\|\varphi\|_{L^2(\O_1)}\\
& \leq C e^{-\frac{2\pi r}{\ell}}\frac{\ell\sqrt{\ell}}{r}\|\varphi\|_{L^2}.
\end{split}
\end{align}
We used the exponential decay at infinity of $U$ by \eqref{wy1}, $r \sim \ell\ln \ell$;  $\|\varphi\|_{L^2} = \sqrt{\ell}\|\varphi\|_{L^2(\Omega_1)}$, $|S^1 - S^k| \geq |S^1 - S^2|$, and
\begin{align*}
\sum\limits_{k=2}^{\ell}e^{-|S^k-S^1|}\leq Ce^{-\frac{2\pi r}{\ell}}.
\end{align*}
For the last estimate, see Lemma 3.5 and Corollary 3.6 in \cite{D-W-Y}. Here and after, $\|\cdot\|_{L^p}$ means always $\|\cdot\|_{L^p(\R^3)}$. Similarly, there holds
\begin{eqnarray}\label{eqs3.7}
\Big|\int_{\R^3}\Bigl(\sum\limits_{k=1}^\ell V_{S^k}^3-V_r^3\Bigr)\psi dx\Big|
\leq C e^{-\frac{2\pi r}{\ell}}\frac{\ell\sqrt{\ell}}{r}\|\psi\|_{L^2}.
\end{eqnarray}
Using $(U,V)=(\alpha W^*,\gamma W^*)$, there holds also, for $\ell$ large,
\begin{align}\label{eqs3.8}
\begin{split}
& \quad \Bigl|\int_{\R^3}\Bigl(\sum\limits_{k=1}^\ell V_{S^k}^2U_{S^k}-V_r^2U_r\Bigr)\varphi dx + \int_{\R^3}\Bigl(\sum\limits_{k=1}^\ell U_{S^k}^2V_{S^k}-U_r^2V_r\Bigr)\psi dx\Bigr|\\
& \leq C \int_{\R^3}\Big(\sum\limits_{i\neq k}V_{S^i}^2U_{S^k}|\varphi|+\sum\limits_{i\neq k}V_{S^i}U^2_{S^k}|\psi|\Big) dx
\\
& \leq
 C e^{-\frac{2\pi r}{\ell}}\frac{\ell\sqrt{\ell}}{r}\big(\|\varphi\|_{L^2}+ \|\psi\|_{L^2}\big).
\end{split}
\end{align}
Using similar estimates and the symmetry of $W_\rho$, we obtain
\begin{align}\label{eqs3.10}
\begin{split}
\Big|\int_{\R^3}U_r W_\rho^2\xi dx\Big| &
= \ell \Big|\int_{\O_1}\Bigl(\sum\limits_{k=1}^\ell U_{S^k}\Bigr)\Bigl(\sum\limits_{k=1}^\ell W_{T^k}\Bigr)^2\xi dx\Big|\\
&\leq C\ell \int_{\Omega_1}\Big[U_{S^1}W_{T^1}^2 + W_{T^1}^2\sum\limits_{k=2}^\ell U_{S^k}+U_{S^1}\sum\limits_{k=2}^{\ell -1}W_{T^k}^2\Big] |\xi| dx \\
&\leq C\ell \int_{\Omega_1}\Big(W_{T^1}\frac{e^{-|S^1-T^1|}}{|S^1-T^1|} +U_{S^1}\sum\limits_{k=2}^{\ell-1} e^{-|S^k-T^1|}\Big)  |\xi| dx\\
&\leq C\ell \frac{e^{-|S^1 - T^1|}}{|S^1 - T^1|}\|\xi\|_{L^2(\Omega_1)}\\
& \leq C e^{-|\rho-re^{\frac{i\pi}{\ell}}|}\frac{\ell\sqrt{\ell}}{r}\|\xi\|_{L^2}.
\end{split}
\end{align}
Similarly,
\begin{align}\label{eqs3.9}
\Big|\int_{\R^3}U_r^2W_\rho\xi dx\Big| \leq C  e^{-|\rho-re^{\frac{i\pi}{\ell}}|}\frac{\ell\sqrt{\ell}}{r}\|\xi\|_{L^2}.
\end{align}
Combining \eqref{eqs3.5}, \eqref{eqs3.6}-\eqref{eqs3.9}, we conclude the estimate for $L_1(\varphi,\psi,\xi)$, hence for $\|L_1\|$.
\end{proof}

By similar considerations, we can claim the following estimate for $L_2$.
\begin{lemma}\label{lm2.1}
There is a constant $C>0$, independent of $\ell$ large, such that for any $r\in D_\ell$,
$$
\big| L_2(\varphi, \psi, \xi) \big|\leq C\|(\varphi, \psi, \xi) \|^2,\quad  \forall \; \|(\varphi, \psi, \xi) \| \leq 1.
$$
\end{lemma}

\section{A key lemma}
Here we will prove that for $\ell$ large enough, the quadratic form $L_2$ is uniformly coercive over the space of perturbation $E$, defined in \eqref{spaceE}.
\begin{lemma}\label{lm2.2}
Let $\beta_{12}\in\Lambda$ in \eqref{lambda}. There is a constant $\delta >0$, independent of $\ell$ large, such that for any $(r, \rho) \in D_\ell$,
$$
L_2(\varphi,\psi,\xi)\geq \delta \|(\varphi,\psi,\xi)\|^2,\quad \forall\; (\varphi,\psi,\xi)\in E.
$$
\end{lemma}
The above estimate is a crucial argument for our setting. In fact, let us denote by $\L$ the linear endomorphism of $E$, associated to $L_2$, i.e. for any $(\varphi,\psi,\xi), (\varphi',\psi',\xi') \in E$,
$$\big\langle \L(\varphi,\psi,\xi), (\varphi',\psi',\xi')\big\rangle_H = \frac{L_2(\varphi+\varphi', \psi+\psi',\xi+\xi') - L_2(\varphi,\psi,\xi) - L_2(\varphi',\psi',\xi')}{2}.$$
Lemma \ref{lm2.2} means that $\L$ is uniformly invertible for $\ell$ large.

\begin{proof}
Suppose the contrary, there exist $\ell\to \infty$, $(r_\ell, \rho_\ell) \in D_\ell$ and $(\varphi_\ell,\psi_\ell,\xi_\ell)\in E$ satisfying
\begin{align}\label{wy3-1}
L_2(\varphi_\ell,\psi_\ell,\xi_\ell) =o(1)\|(\varphi_\ell,\psi_\ell,\xi_\ell)\|^2.
\end{align}
Without loss of generality, we may assume $\|(\varphi_\ell,\psi_\ell,\xi_\ell)\|^2_{H^1(\R^3)}=\ell$. Notice that due to the limiting problem, here we prefer to normalize with the standard $H^1$ norm instead of norm in $H$, even they are equivalent. For simplicity, we erase the index $\ell$ for $ r_\ell$. For $1 \leq k \leq \ell$, let $\O_k$ be given by \eqref{Omegak} and
\begin{align}
\label{tOmegak}
\widetilde{\Omega}_k=\Big\{z=(z' ,z_3)\in\R^2\times\R:\Big\langle\frac{z' }{|z' |},\frac{y'^k}{|y'^k|}\Big\rangle\geq\cos\frac{\pi}{\ell}\Big\}.
\end{align}
By symmetry,  there holds
\begin{eqnarray}\label{eqswqf3.2}
\|(\varphi_\ell,\psi_\ell,\xi_\ell)\|^2_{H^1(\Omega_1)}= \|(\varphi_\ell,\psi_\ell,\xi_\ell)\|^2_{H^1(\widetilde \Omega_1)} = 1.
\end{eqnarray}
With abuse of notation, we denote still by $\L$ the linear endomorphism associated to $L_2$ over $H$. For any $(g,h,f)\in H$, we have
\begin{align}
\label{eqs3.1}
\begin{split}
J_\ell & = \frac{1}{\ell} \langle \L(\varphi_\ell,\psi_\ell,\xi_\ell),(g,h,f) \rangle_H\\ & = \int_{\Omega_1}\Bigl(\nabla\varphi_\ell\nabla g+P_1 \varphi_\ell g-3\mu_1 U_r^2g\varphi_\ell\Bigr)dx
+\int_{\Omega_1}\Bigl(\nabla\psi_\ell\nabla h+P_2 h\psi_\ell-3\mu_2 V_r^2h\psi_\ell\Bigr)dx\\
& \quad +\int_{\Omega_1}\Bigl(\nabla\xi_\ell\nabla f+P_3 f\xi_\ell-3\mu_3 W_\rho^2f\xi_\ell\Bigr)dx\\
&
\quad -\beta_{12}\int_{\Omega_1}\Bigl(U_r^2h\psi_\ell+V_r^2g\varphi_\ell+2U_r V_r g\psi_\ell+2U_r V_rh\varphi_\ell\Bigr)dx\\
& \quad
-\beta_{13}\int_{\Omega_1}\Bigl(U_r^2f\xi_\ell+W_\rho^2g\varphi_\ell+2U_r W_\rho g\xi_\ell+2U_r W_\rho f\varphi_\ell\Bigr)dx\\
& \quad
-\beta_{23}\int_{\Omega_1}\Bigl(V_r^2f\xi_\ell+W_\rho^2h\psi_\ell+2V_r W_\rho h\xi_\ell+2V_r W_\rho f\psi_\ell \Bigr)dx
\\
& \quad
= o\Big(\frac{1}{\sqrt{\ell}}\Big)\|(g,h,f)\|.
\end{split}
\end{align}
Let
$$
\overline{\varphi}_\ell=\varphi_\ell(x-S^1),\ \ \overline{\psi}_\ell=\psi_\ell(x-S^1),\ \ \overline{\xi}_\ell=\xi_\ell(x-T^1).
$$
Given any $R>0$, we see that $B_R(S^1)\subset\Omega_1$ for $\ell$ large since $r_\ell \sim \ell\ln \ell$. Thus \eqref{eqswqf3.2}
implies that for $\ell$ large enough,
\begin{align*}
\left\|\overline{\varphi}_\ell\right\|^2_{H^1(B_R)} + \left\|\overline{\psi}_\ell\right\|^2_{H^1(B_R)} + \left\|\overline{\xi}_\ell\right\|^2_{H^1(B_R)}\leq 1.
\end{align*}
Up to a subsequence, we may assume the existence of $\varphi,\psi, \xi\in H^1(\R^3)$ such that as $\ell$ tends to infinity,
\begin{align}
\label{limit1}
(\overline{\varphi}_\ell, \overline{\psi}_\ell, \overline{\xi}_\ell) \longrightarrow (\varphi, \psi, \xi) \ \ \mbox{weakly in }\  H_{loc}^1(\R^3)^3 \ \text{ and strongly in }\   L_{loc}^2(\R^3)^3.
\end{align}
Moreover, $\varphi$ and $\psi$ are even in $x_2, x_3$; satisfy
\begin{align}\label{wy2}
\int_{\R^3}W^{*2}\Big(\frac{\partial U}{\partial x_1}\varphi+\frac{\partial V}{\partial x_1}\psi \Big)dx=0.
\end{align}
To get the above orthogonality condition, we used $(\varphi_\ell, \psi_\ell, \xi_\ell) \in E$, the $L^2_{loc}$ convergence of $(\varphi_\ell, \psi_\ell, \xi_\ell)$, also the exponential decay at infinity of $W^*$ and its derivatives.

\smallskip
We claim that $(\varphi,\psi)$ satisfies the linearized equation for $(U, V)$, that is
\begin{eqnarray}\begin{cases}\label{eqs3.3}
 -\Delta \varphi + \varphi-3\mu_1 U^2\varphi-\beta_{12}V^2\varphi-2\beta_{12}UV\psi=0 ,\cr
 -\Delta \psi + \psi-3\mu_2 V^2\psi-\beta_{12}U^2\psi-2\beta_{12}UV\varphi=0
 \end{cases}
\end{eqnarray}
in $\R^3$.
Let $(g,h)\in C_0^\infty(\R^3)^2$,
and $(g_\ell,h_\ell,0)=(g_{S^1},h_{S^1},0)$, so $(g_\ell, h_\ell)\in C_0^\infty(B_R(S^1))^2$ for some $R > 0$.
It's not difficult to see that
$$
\int_{\Omega_1}\Bigl(\nabla\varphi_\ell\nabla g_\ell+g_\ell\varphi_\ell
-3\mu_1 U_r^2g_\ell\varphi_\ell\Bigr)dx\longrightarrow\int_{\R^3}\Bigl(\nabla\varphi\nabla g+g\varphi-3\mu_1 U^2g\varphi\Bigr)dx,
$$
$$
\int_{\Omega_1}\Bigl(\nabla\psi_\ell\nabla h_\ell+h_\ell\psi_\ell
-3\mu_2 V_r^2h_\ell\psi_\ell\Bigr)dx\longrightarrow\int_{\R^3}\Bigl(\nabla\psi\nabla h+h \psi-3\mu_2 V^2h\psi\Bigr)dx.
$$
Similarly, there hold
\begin{align*}
& \quad \quad \int_{\Omega_1}\Bigl(U_r^2h_\ell\psi_\ell+V_r^2g_\ell\varphi_\ell+2U_r V_r g_\ell\psi_\ell
+2U_r V_rh_\ell\varphi_\ell\Bigr)dx\cr
&\longrightarrow \int_{\R^3}\Bigl(U^2h\psi+V^2g\varphi+2U V g\psi+2U Vh\varphi\Bigr)dx,
\end{align*}
and
\begin{align*}
-\beta_{13}\int_{\Omega_1}\Bigl(W_\rho^2g\varphi_\ell+2U_r W_\rho g\xi_\ell \Bigr)dx
-\beta_{23}\int_{\Omega_1}\Bigl(W_\rho^2h\psi_\ell+2V_r W_\rho h\xi_\ell\Bigr)dx \longrightarrow 0.
\end{align*}
Combining with \eqref{eqs3.1}, we get
\begin{align}\label{eqs3.2}
\begin{split}
& \quad
\int_{\R^3}\Bigl(\nabla\varphi\nabla g+g\varphi-3\mu_1 U^2g\varphi\Bigr)dx+\int_{\R^3}\Bigl(\nabla\psi\nabla h+h\psi-3\mu_2 V^2h\psi\Bigr)dx\\
& -\beta_{12}\int_{\R^3}\Bigl(U^2h\psi+V^2g\varphi+2U V g\psi+2UVh\varphi\Bigr)dx=0.
\end{split}
\end{align}
By the density of $C_0^\infty(\R^3)$ in $H^1(\R^3)$, \eqref{eqs3.2} holds true for all $(g, h) \in H^1(\R^3)^2$, which means that  \eqref{eqs3.3} is satisfied in the weak sense.

\smallskip
Since $(U,V)$ is nondegenerate, and we work with functions even in $ x_2,x_3$, the kernel of \eqref{eqs3.3} is given by $(\frac{\alpha}{\gamma}\frac{\partial W^*}{\partial x_1},\frac{\partial W^*}{\partial x_1})$. In other words, $(\varphi,\psi) = c(\frac{\partial U}{\partial x_1},\frac{\partial V}{\partial x_1})$ for some $c \in \R$. Applying \eqref{wy2}, we deduce that $(\varphi,\psi)=(0,0)$. Consequently, for any $R > 0$,
\begin{align}\label{wy3-2}
\lim_{\ell \to \infty} \int_{B_R(S^1)}(\varphi_\ell^2+\psi_\ell^2)dx = 0.
\end{align}

By the same, let $f \in C_0^\infty(\R^3)$ and $(0,0,f_{T_1})$ be the test function in \eqref{eqs3.1}, using \eqref{limit1},
we can claim that $
-\Delta \xi + \xi-3\mu_3 W^2\xi=0$ in $\R^3$, and $\langle W_{T^1}^{*2}Z_1, \xi \rangle_{L^2(\R^3)} =0$.
With the nondegeneracy of $W$, we find $ \xi=0$ so that
\begin{align}\label{wy3-3}
\lim_{\ell \to \infty} \int_{B_R(T^1)} \xi_\ell^2dx = 0, \quad \forall\; R > 0.
\end{align}
Furthermore, there holds
\begin{align}\label{wy3-4}
U_{r}\leq Ce^{-\frac{|x-S^1|}{2}},\; V_{r}\leq Ce^{-\frac{|x-S^1|}{2}}\ \ \mbox{in }\; \Omega_1; \quad
W_{\rho}\leq Ce^{-\frac{|x-T^1|}{2}}\ \ \mbox{in }\; \widetilde{\Omega}_1.
\end{align}
From \eqref{wy3-1}, \eqref{wy3-2}--\eqref{wy3-4} and the symmetry, we see that
\begin{align}\label{eqs3.4}
\begin{split}
o(1)\ell
&=
\|(\varphi_\ell,\psi_\ell,\xi_\ell)\|^2-\ell\int_{\Omega_1}\Big(3\mu_1 U_r^2\varphi_\ell^2+3\mu_2 V_r^2\psi_\ell^2\Big)dx
-3\mu_3 \ell\int_{\widetilde{\Omega}}W_\rho^2\xi_\ell^2dx
\\
&\quad
-\beta_{12}\ell\int_{\Omega_1}\Big( U_r^2\psi_\ell^2+V_r^2\varphi_\ell^2+4U_r V_r \varphi_\ell \psi_\ell \Big)dx\\
& \quad -\beta_{13}\int_{\R^3}\Big( U_r^2\xi_\ell^2+W_\rho^2\varphi_\ell^2+4U_r W_\rho \varphi_\ell\xi_\ell \Big)dx
\\
& \quad
-\beta_{23}\int_{\R^3}\Big( V_r^2\xi_\ell^2+W_\rho^2\psi_\ell^2+4V_r W_\rho \psi_\ell\xi_\ell \Big)dx
\\
&=
\ell+o(\ell)+o\Big(e^{-R}\Big) \times \int_{\Omega_1} \Big( \varphi_\ell^2+\psi_\ell^2 \Big)dx
+o\Big(e^{-R} \Big) \times \int_{\widetilde{\Omega}_1}\xi_\ell^2dx\\
&
\quad +O\Big(e^{-|S^1-T^1|}\Big)\|\varphi_\ell\|_{P_1}\|\xi_\ell\|_{P_3}-\beta_{13}\int_{\R^3}\Big(U_r^2\xi_\ell^2+W_\rho^2\varphi_\ell^2\Big)dx\\
& \quad -\beta_{23}\int_{\R^3}(V_r^2\xi_\ell^2+W_\rho^2\psi_\ell^2)dx
\\
&\geq
\ell-C(|\beta_{13}|+|\beta_{23}|)\ell+o(e^{-R})\ell.
\end{split}
\end{align}
Here we used
\begin{align*}
\int_{\R^3}\Big(U_r^2\xi_\ell^2+W_\rho^2\varphi_\ell^2\Big)dx + \int_{\R^3}\Big(V_r^2\xi_\ell^2+W_\rho^2\psi_\ell^2\Big)dx \leq C\int_{\R^3} \Big(\xi_\ell^2+\varphi_\ell^2 \Big)dx.
\end{align*}
However, if $|\beta_{13}|+|\beta_{23}|$ is small enough, the estimate \eqref{eqs3.4} is impossible for large $R$ and $\ell$. Thus we reach a contradiction, so the hypothesis was wrong, hence the lemma holds true.
\end{proof}

\section{Energy expansion for ansatz}
Here we will expand the energy $I(U_r,V_r,W_\rho)$.
Recall that
\begin{align*}
 I(u,v,w)
 &= \frac{\|(u,v,w)\|^2}{2} -
\frac{1}{4}\ds\int_{\R^3}(\mu_1 u^4+\mu_2 v^4+\mu_3w^4)dx
\\
& \quad -
\frac{1}{2}\ds\int_{\R^3}(\beta_{12}u^2v^2+\beta_{13}u^2w^2+\beta_{23}v^2w^2)dx.
\end{align*}
\begin{proposition} \label{A.1}
Let $M > 0$, there holds, for any $\ell$ be large enough and $(r, \rho) \in D_\ell$,
\begin{align*}
I(U_r,V_r,W_\rho)
&=
\ell A_0+\ell\Bigl(\frac{a_1\alpha^2}{r^{m_1}}+\frac{a_2\gamma^2}{r^{m_2}}+\frac{a_3}{\mu_3 \rho^{m_3}}\Bigr)A_1-(\beta_{12}A_2+A_3)e^{-\frac{2\pi r}{\ell}}\frac{\ell^2}{r}- A_4 e^{-\frac{2\pi \rho}{\ell}}\frac{\ell^2}{\rho}
\\
&\quad
- (|\beta_{13}| + |\beta_{23}|)o\Big(e^{-|\rho-re^{\frac{i\pi}{\ell}}|}\frac{\ell^2}{r}\Big) +O\Bigl(\frac{\ell^{1 - \sigma}}{r^{\min(m_1, m_2)}}+\frac{\ell^{1 - \sigma}}{\rho^{m_3}}\Bigr) + O\Big(\ell e^{-\frac{3\pi r}{\ell}}\Big)
\end{align*}
where
\begin{align}
\label{A01}
A_0=\left[\frac{\mu_1+\mu_2-2\beta_{12}}{4(\mu_1\mu_2-\beta_{12}^2)} +\frac{1}{4\mu_3}\right] \int_{\R^3}W^{*4}dx, \quad A_1=\frac{1}{2}\ds\int_{\R^3}W^{*2}dx
\end{align}
and $A_2, A_3, A_4$ are also positive constants dependent only on $\mu_i$, $\alpha$, $\gamma$ and $W^*$.
\end{proposition}


\begin{proof}
In fact, there holds
\begin{align*}
& \quad I(U_r,V_r,W_\rho)\\
&=
\ell\Bigl(\frac{\|U\|^2_{H^1(\R^3)} + \|V\|^2_{H^1(\R^3)}}{2}
-\frac{1}{4}\ds\int_{\R^3}(\mu_1 U^4+\mu_2V^4)dx -\frac{1}{2}\ds\int_{\R^3}\beta_{12}U^2V^2dx + \frac{\mu_3}{4}\int_{\R^3}W^4dx\Bigr)\\
& \quad +\frac{1}{2}\int_{\R^3}(P_1 -1)U_r^2dx+\frac{1}{2}\int_{\R^3}(P_2-1)V_r^2dx
+\frac{1}{2}\int_{\R^3}(P_3 -1)W_\rho^2dx\\
&\quad -
\frac{\beta_{12}}{2}\int_{\R^3}\Big(U_r^2V_r^2-\sum\limits_{k=1}^\ell U_{S^k}^2V_{S^k}^2
-\sum\limits_{k\neq i}V_{S^k}^2U_{S^k}U_{S^i}-\sum\limits_{k\neq i}V_{S^k}U_{S^k}^2V_{S^i}\Big)dx
\\
& \quad -
\frac{1}{2}\int_{\R^3}\Big(\beta_{13}U_r^2W_\rho^2dx+\beta_{23}V_r^2W_\rho^2\Big)dx\\
& \quad -
\frac{\mu_1}{4}\int_{\R^3}\Big(U_r^4-\sum\limits_{k=1}^\ell U^4_{S^k} - 2\sum\limits_{k\neq i}U_{S^k}^3U_{S^i}\Big)dx -
\frac{\mu_2}{4}\int_{\R^3}\Big(V_r^4-\sum\limits_{k=1}^\ell V_{S^k}^4-2\sum\limits_{j\neq i}V_{S^k}^3V_{S^i}\Big)dx\\
& \quad -
\frac{\mu_3}{4}\int_{\R^3}\Big(W_\rho^4-\sum\limits_{k=1}^\ell {W_{T^k}}^4-2\sum\limits_{j\neq i}W_{T^k}^3W_{T^i}\Big)dx
\\
&=: \ell I_1+I_2+I_3+I_4+I_5.
\end{align*}
Here we used
$$U_r^2 = \sum_k U_{S^k}^2 + \sum_{k\ne i}U_{S^k}U_{S^i}$$ and similar rewriting for $V_r^2$. On the other hand, we expressed the integrals of mixed products by using the equations of $U_{S^k}$ or $V_{S_k}$ from \eqref{eqs1.3}. Now we will develop more precisely all the terms $I_j$.
\begin{align*}
I_1
& =\frac{\|U\|^2_{H^1(\R^3)} + \|V\|^2_{H^1(\R^3)}}{2}
-\frac{1}{4}\ds\int_{\R^3}(\mu_1 U^4+\mu_2V^4)dx -\frac{1}{2}\ds\int_{\R^3}\beta_{12}U^2V^2dx + \frac{\mu_3}{4}\int_{\R^3}W^4dx\\
& = \left[\frac{1}{4\mu_3^2} +\frac{\mu_1+\mu_2-2\beta_{12}}{4(\mu_1\mu_2-\beta_{12}^2)}\right] \int_{\R^3}{W^{*}}^4dx
\end{align*}
By symmetry and the exponential decay of $W^*$, we obtain
\begin{align*}
\int_{\R^3}(P_1-1)U_r^2dx
= \ell\int_{\Omega_1}(P_1 -1) U_r^2dx & =
\ell\int_{\Omega_1}(P_1-1)\Bigl[U_{S^1}^2+2U_{S^1}\sum\limits_{i=2}^\ell U_{S^i}+\Big(\sum\limits_{i=2}^\ell U_{S^i}\Big)^2\Big]dx
\\
&=
\ell\Bigl[\frac{a_1\alpha^2}{r^{m_1}}\int_{\R^3}{W^*}^2dx+O\Bigl(\frac{1}{r^{m_1}\ell^\sigma}\Bigr)\Bigr]
\end{align*}
where $\sigma>0$ is the constant in assumption $(H_m)$ or $(\widetilde H_m)$, which can be as small as we want. Similarly,
$$
\int_{\R^3}(P_2 -1)V_r^2dx=\ell\Bigl[\frac{a_2\gamma^2}{r^{m_2}}\int_{\R^3}{W^*}^2dx+O\Big(\frac{1}{r^{m_2}\ell^\sigma}\Big)\Big],
$$
and
$$
\int_{\R^3}(P_3 -1)W_\rho^2dx=\ell\Bigl[\frac{a_3}{\mu_3\rho^{m_3}}\int_{\R^3}{W^*}^2dx+O\Big(\frac{1}{\rho^{m_3}\ell^\sigma}\Big)\Big].
$$
Then
\begin{align*}
I_2=\frac{\ell}{2}\Bigl(\frac{a_1\alpha^2}{r^{m_1}} +\frac{a_2\gamma^2}{r^{m_2}} +\frac{a_3}{\mu_3\rho^{m_3}}\Big) \int_{\R^3}{W^*}^2dx + O\Big(\frac{1}{\rho^{m_3}\ell^\sigma}\Big) + O\Big(\frac{1}{r^{\min(m_1, m_2)}\ell^\sigma}\Big).
\end{align*}
By the same idea, we can claim
\begin{align*}
I_3 & = -
\frac{\beta_{12}}{2}\int_{\R^3}\Big(U_r^2V_r^2-\sum\limits_{k=1}^\ell U_{S^k}^2V_{S^k}^2
-\sum\limits_{k\neq i}V_{S^k}^2U_{S^k}U_{S^i}-\sum\limits_{k\neq i}V_{S^k}U_{S^k}^2V_{S^i}\Big)dx\\
& = -
\frac{\beta_{12}\ell}{2}\int_{\Omega_1}\Big(U_r^2V_r^2-\sum\limits_{k=1}^\ell U_{S^k}^2V_{S^k}^2
-\sum\limits_{k\neq i}V_{S^k}^2U_{S^k}U_{S^i}-\sum\limits_{k\neq i}V_{S^k}U_{S^k}^2V_{S^i}\Big)dx\\
& = -\frac{\beta_{12}\ell}{2}\int_{\Omega_1} \Big(U_{S^1}^2V_{S^1}\sum\limits_{i=2}^\ell V_{S^i}+V_{S^1}^2U_{S^1}\sum\limits_{i=2}^\ell U_{S^i}\Big) + O\Big(\ell^2 e^{-\frac{3\pi r}{\ell}}\Big).
\end{align*}
Since
\begin{eqnarray*}
\int_{\R^3}U_r^2W_\rho^2dx
=
\ell\int_{\Omega_1}U_r^2W_\rho^2dx =  o\Big( \frac{\ell^2}{r}e^{-2|\rho-re^{\frac{i\pi}{\ell}}|}\Big) ,
\end{eqnarray*}
there holds
\begin{align*}
I_4= -\frac{1}{2}\int_{\R^3}\Big(\beta_{13}U_r^2W_\rho^2dx+\beta_{23}V_r^2W_\rho^2\Big)dx =(|\beta_{13}| +|\beta_{23}|)\times o\Big( \frac{\ell^2}{r}e^{-2|\rho-re^{\frac{i\pi}{\ell}}|}\Big).
\end{align*}
At last, we have
\begin{align*}
I_5=-\frac{\mu_1}{4}J_1-\frac{\mu_2}{4}J_2-\frac{\mu_3}{4}J_3,
\end{align*}
where
\begin{align*}
J_1 =\int_{\R^3}\Big(U_r^4-\sum\limits_{k=1}^\ell U^4_{S^k} - 2\sum\limits_{k\neq i}U_{S^k}^3U_{S^i}\Big)dx & = \ell \int_{\Omega_1}\Big(U_r^4-\sum\limits_{k=1}^\ell U^4_{S^k} - 2\sum\limits_{k\neq i}U_{S^k}^3U_{S^i}\Big)dx\\
&=
 2\ell\int_{\Omega_1}U_{S^1}^3\sum\limits_{i = 2}^\ell U_{S^i}dx+O\Bigl(\ell e^{-\frac{3\pi r}{\ell}}\Bigr).
\end{align*}
Using Lemma 2.3 in \cite{LNW}, we can claim that there exists a constant $\xi_1 > 0$ depending on $\alpha$, $W^*$ such that as $\ell$ tends to infinity,
\begin{align*}
\int_{\Omega_1}U_{S^1}^3\sum\limits_{i = 2}^\ell U_{S^i}dx = \xi_1\frac{\ell}{r}e^{-\frac{2\pi r}{\ell}} + O\Bigl(\ell e^{-\frac{3\pi r}{\ell}}\Bigr).
\end{align*}
Analogously, there is $\xi_2 > 0$ depending on $\gamma$, $W^*$ such that
\begin{align*}
J_2=\int_{\R^3} \Big(V_r^4-\sum\limits_{k=1}^\ell V_{S^k}^4-2\sum\limits_{j\neq i}V_{S^k}^3V_{S^i}\Big)dx & =2\ell \int_{\Omega_1}V_{S^1}^3\sum\limits_{i=2}^\ell V_{S^i}dx+O\Bigl(\ell e^{-\frac{3\pi r}{\ell}}\Bigr)\\
& = 2\xi_2\frac{\ell^2}{r}e^{-\frac{2\pi r}{\ell}} + O\Bigl(\ell e^{-\frac{3\pi r}{\ell}}\Bigr)
\end{align*}
and we have $\xi_3 > 0$ depending on $\mu_3$, $W^*$ satisfying
\begin{align*}
J_3 =\int_{\R^3}\Big(W_\rho^4-\sum\limits_{k=1}^\ell {W_{T^k}}^4-2\sum\limits_{j\neq i}W_{T^k}^3W_{T^i}\Big)dx & = 2\ell \int_{\widetilde{\Omega}_1}W_{T^1}^3\sum\limits_{i=2}^\ell W_{T^i}dx+O\Big(\ell e^{-\frac{3\pi \rho}{\ell}}\Big)\\
& = 2\xi_3\frac{\ell^2}{\rho}e^{-\frac{2\pi \rho}{\ell}} + O\Big(\ell e^{-\frac{3\pi \rho}{\ell}}\Big).
\end{align*}
Clearly, similar estimate holds true for $I_3$ since $V_{S^1}= \gamma U_{S^1}/\alpha$. Finally, we arrive at
\begin{align*}
I_3+I_5 & =I_3-\frac{\mu_1}{4}J_1-\frac{\mu_2}{4}J_2-\frac{\mu_3}{4}J_3\\
& =-(\beta_{12}A_2+ A_3)\frac{\ell^2}{r}e^{-\frac{2\pi r}{\ell}}-A_4\frac{\ell^2}{\rho}e^{-\frac{2\pi\rho}{\ell}}+O\Big(\ell e^{-\frac{3\pi r}{\ell}}\Big) + O\Big(\ell e^{-\frac{3\pi \rho}{\ell}}\Big)
\end{align*}
where $A_j$ are positive constants depending only on $\mu_i$, $\alpha$, $\gamma$ and $W^*$.

\medskip
Combining all the estimates for $I_j$ ($1 \leq j \leq 5$), we complete readily the proof.
\end{proof}

\section{The finite-dimensional reduction}
In this section, we intend to prove main theorems by the Lyapunov-Schmidt reduction method. Recall that $M > 0$ is the constant in the definition of $D_\ell$ to be chosen later, see \eqref{defr}.

\begin{proposition} \label{pro2.6}
Given $M > 0$, there is an integer $\ell_0>0$ such that for each $\ell\geq \ell_0$, there is a $C^1$ map $\Gamma$ from $D_\ell$ to $E$ with $r=|S^1|,\rho=|T^1|$, and $(\varphi,\psi,\xi) = \Gamma(r, \rho) \in E$ satisfies
$$
\langle dJ(\varphi,\psi,\xi), (\varphi',\psi',\xi')\rangle_H = 0, \quad \forall\; (r, \rho) \in D_\ell, \; (\varphi',\psi',\xi') \in E.
$$
Moreover, there is a positive constant $M_0 > 0$ such that
$$
\|(\varphi,\psi,\xi)\|\leq  M_0\Bigl(\frac{\ell}{r^{\min(m_1, m_2)}}+\frac{\ell}{\rho^{m_3}}+(|\beta_{13}|+|\beta_{23}|)e^{-|\rho-re^{\frac{i\pi}{\ell}}|}\frac{\ell\sqrt\ell}{r}\Bigr)
$$
\end{proposition}
\begin{proof}
Recall that $L_1$ is a bounded linear functional over $H$, ${\mathcal L}$ denotes the linear endomorphism of $E$ associated to $L_2$. Therefore, to find a critical point for $J(\varphi,\psi,\xi) $ over $E$ is equivalent to solve in $E$ the following equation:
\begin{eqnarray}\label{eqs3.11}
L_1+ \L(\varphi,\psi,\xi) +R' (\varphi,\psi,\xi)=0.
\end{eqnarray}
Since $\L$ is invertible over $E$ by Lemma \ref{lm2.2}, \eqref{eqs3.11} can be written as
$$
(\varphi,\psi,\xi)=\A(\varphi,\psi,\xi)=-\L^{-1}L_1- \L^{-1}R' (\varphi,\psi,\xi).
$$
Set
\begin{align*}
\Sigma = \Bigl\{ (\varphi,\psi,\xi)\in E,  \|(\varphi,\psi,\xi)\|\leq M_0\e_0\Bigr\}
\end{align*}
where
$$\e_0 := \frac{\ell}{r^{\min(m_1, m_2)}}+\frac{\ell}{\rho^{m_3}}+(|\beta_{13}|+|\beta_{23}|) e^{-|\rho-re^{\frac{i\pi}{\ell}}|}\frac{\ell\sqrt\ell}{r}$$
and $M_0$ is a large enough positive constant to be chosen.

\medskip
Notice that $\e_0 \to 0$ when $\ell \to \infty$. From Lemma \ref{lm2.4} and  Lemma \ref{lm2.5}, for any $(\varphi, \psi, \xi) \in \Sigma$, there holds
\begin{align}\label{eqs3.12}
\begin{split}
\|\A(\varphi,\psi,\xi)\|
\leq C\|L_1\|+C\|(\varphi,\psi,\xi)\|^2 & \leq C(\e_0 + M_0^2\e_0^2) \\
& = C(1+ M_0^2\e_0)\e_0 \\
& \leq M_0\e_0
\end{split}
\end{align}
if $M_0 > C$ and $\ell$ is large enough. Moreover, given any $(\varphi_1,\psi_1,\xi_1), (\varphi_2,\psi_2,\xi_2) \in \Sigma$,
\begin{align}\label{eqs3.13}
\begin{split}
\|\A(\varphi_1,\psi_1,\xi_1)- \A(\varphi_2,\psi_2,\xi_2)\|
& =\|\L^{-1}R' (\varphi_1,\psi_1,\xi_1)- \L^{-1}R' (\varphi_2,\psi_2,\xi_2)\|\\ & \leq
CM_0\e_0\|(\varphi_1,\psi_1,\xi_1)-(\varphi_2,\psi_2,\xi_2)\|.
\end{split}
\end{align}
Therefore, $\A$ is a contraction from $\Sigma$ into itself for $\ell$ large. By Banach's fixed point theorem, there is $(\varphi,\psi,\xi)\in \Sigma$ such that $\A(\varphi,\psi,\xi)=(\varphi,\psi,\xi)$, hence \eqref{eqs3.11}. The proof is completed.
\end{proof}

\medskip
Now we are ready to prove the existence of simultaneous synchronized and segregated solutions to system \eqref{eq1}.
\subsection{Proof of Theorem \ref{th1.1}}
Let $ (\varphi_{r, \rho},\psi_{r, \rho},\xi_{r, \rho})= \Gamma(r,\rho)$ be the mapping obtained in Proposition \ref{pro2.6}. Define
$$
F(r, \rho)= J\circ \Gamma(r, \rho) = I(U_r+\varphi_{r, \rho},V_r+\psi_{r, \rho},W_\rho+\xi_{r, \rho}), \quad \forall\, (r,\rho)\in D_\ell.
$$
With similar arguments used in Proposition 3 of \cite{O} (see also \cite{CNY}), we can check that for $\ell$ sufficiently large, if $(r,\rho) \in D_\ell$ is a critical point of $F$, then $(U_r,V_r,W_\rho) + \Gamma(r,\rho) \in H$ is a critical point of $I$.

\medskip
Consider first the case $(ii)$ of $(H_m)$. For simplicity, denote $m := m_3=\min\{m_1,m_2\}$. Applying Lemma \ref{lm2.4}, Lemma \ref{lm2.1} and  Proposition \ref{A.1}, for $\ell$ large enough and any $(r, \rho) \in D_\ell$, we have
\begin{align*}
F(r,\rho)
&= I(U_r,V_r,W_\rho)+L_1(\Gamma(r,\rho))+\frac{L_2(\Gamma(r,\rho))}{2}
+R(\Gamma(r,\rho))
\\
&=
I(U_r,V_r,W_\rho)+O\Bigl(\|L_1\|\|\Gamma(r,\rho)\|+\|\Gamma(r,\rho)\|^2\Bigr)
\\
&=
I(U_r,V_r,W_\rho)+(|\beta_{13}|+\beta_{23}|)^2\times O\Big(e^{-2|\rho-re^{\frac{i\pi}{\ell}}|}\frac{\ell^3}{r^2}\Big)
+O\Bigl(\frac{\ell^2}{r^{2m}}\Bigr)
\\
&=
\ell A_0+\ell\Bigl(\frac{a_1\alpha^2}{r^{m_1}}+\frac{a_2\gamma^2}{r^{m_2}}+\frac{a_3}{\mu_3 \rho^{m}}\Bigr)A_1-(\beta_{12}A_2+A_3)e^{-\frac{2\pi r}{\ell}}\frac{\ell^2}{r}- A_4 e^{-\frac{2\pi \rho}{\ell}}\frac{\ell^2}{\rho}
\\
&\quad
- (|\beta_{13}| + |\beta_{23}|)\times o\Big(e^{-|\rho-re^{\frac{i\pi}{\ell}}|}\frac{\ell^2}{r}\Big) +O\Bigl(\frac{\ell}{r^{m+\sigma}}+\frac{\ell}{\rho^{m+\sigma}}\Bigr) + O\Bigl(\frac{\ell^2}{r^{2m}}\Bigr)\cr
&
\quad +(|\beta_{13}|+\beta_{23}|)^2\times O\Big(e^{-2|\rho-re^{\frac{i\pi}{\ell}}|}\frac{\ell^3}{r^2}\Big)
 + O\Big(\ell e^{-\frac{3\pi r}{\ell}}\Big)
\cr
&=
\ell A_0+\ell\Bigl(\frac{a_1\alpha^2+a_2\gamma^2}{r^{m}}\Bigr)A_1+\ell\frac{a_3}{\mu_3 \rho^{m}}A_1- e^{-\frac{2\pi r}{\ell}}\frac{\ell^2}{r}(\beta_{12}A_2+A_3) - e^{-\frac{2\pi \rho}{\ell}}\frac{\ell^2}{\rho}A_4
\cr
&\quad  -
K(\beta)\frac{\ell^2}{r}e^{-2|\rho-re^{\frac{i\pi}{\ell}}|}+O\Bigl(\frac{\ell}{ r^{m+\sigma}}\Bigr) + O\Big(\ell e^{-\frac{3\pi r}{\ell}}\Big),
\end{align*}
where
$$K(\beta) := o\big(|\beta_{13}|+|\beta_{23}|\big) + O\Big(\frac{(|\beta_{13}|+|\beta_{23}|)^2}{\ln \ell}\Big).$$

\medskip
We claim that the maximal value of $F$ over $D_\ell$ is reached by some interior point of $D_\ell$. The key argument is the following:
\begin{lem}
\label{newlem1}
Under the assumption $(H_m)$, there exist $C_0, M_1 > 0$ depending on $a_j, A_j, m$ such that if $M \geq M_1$, then for $\ell$ large enough.
\begin{align}\label{4}
\max\limits_{(r,\rho)\in D_\ell}F(r,\rho) \geq \ell\Big(A_0+\frac{C_0 + o(1)}{(\ell\ln\ell)^m}\Big).
\end{align}
\end{lem}

Let
$$
g_1(t)=\frac{(a_1 \alpha^2+a_2\gamma^2)A_1}{t^m\ell^m}-\frac{(\beta_{12}A_2+A_3)e^{-2\pi t}}{t}.
$$
Then
$$
g_1'(t)=-\frac{m(a_1 \alpha^2+a_2\gamma^2)A_1}{t^{m+1}\ell^m}+\frac{2\pi (\beta_{12}A_2+A_3)e^{-2\pi t}}{t}+\frac{(\beta_{12}A_2+A_3)e^{-2\pi t}}{t^2},
$$
By Cauchy-Schwarz inequality,  we see that $\beta_{12}A_2+ A_3>0$ for any $\beta_{12}>-\sqrt{\mu_1\mu_2}$. Then it is easy to check that for $\ell \to \infty$, $g_1$ has a large local maximum point $ t_\ell$ satisfying the equation
\begin{eqnarray}\label{eqs3.14}
\frac{m(a_1 \alpha^2+a_2\gamma^2)A_1}{t^{m+1}\ell^m}=\frac{2\pi B_1e^{-2\pi t}}{t}+\frac{B_1e^{-2\pi t}}{t^2}.
\end{eqnarray}
Thus
$$
t_\ell= \Big(\frac{m}{2\pi}+o(1)\Big)\ln \ell.
$$
Moreover, by \eqref{eqs3.14}, we have
\begin{align*}
g_1(t_{\ell})=\frac{(a_1\alpha^2+a_2\gamma^2)A_1}{t_\ell^m\ell^m}\Big(1+O({t_{\ell}^{-1}})\Big).
\end{align*}
So the function $\widetilde{g}_1 r = g_1(r/\ell)$ has a local maximum point $r_\ell=\ell t_\ell=(\frac{m}{2\pi}+o(1))\ell\ln\ell$
with
$$
\widetilde{g}_1(r_\ell)=\frac{C_1+o(1)}{(\ell\ln\ell)^m}, \quad \mbox{as $\ell \to \infty$}
$$
for some constant $C_1>0$ depending only on $a_1, a_2, \alpha, \gamma, m$.

\medskip
Similarly, as $\ell\to \infty$, the function $
\widetilde{h}(\rho)=\frac{a_3A_1}{\mu_3\rho^m}-\frac{A_4\ell}{\rho}e^{-\frac{2\pi \rho}{\ell}}$
has a local maximum point at
$
\rho_\ell=(\frac{m}{2\pi}+o(1))\ell\ln\ell,
$
where
$$
\widetilde{h}(\rho_\ell)=\frac{C_2+o(1)}{(\ell\ln\ell)^m}
$$
for some constant $C_2 > 0$ depending only on $a_3,\,A_4,\,m$.

\medskip
For $M > M_1 = \frac{m}{\pi}$ and $\ell$ large enough, we see that
\begin{align}\label{eqs3.15}
\begin{split}
\max_{D_\ell} F(r,\rho) \geq F(r_\ell, \rho_\ell) & = \ell A_0+ \ell \widetilde{g}(r_\ell)+ \ell \widetilde{h}(\rho_\ell)\\
& \quad  -
K(\beta)\frac{\ell^2}{r_\ell}e^{-2|\rho_\ell -r_\ell e^{\frac{i\pi}{\ell}}|} + O\Bigl(\frac{\ell}{r_\ell^{m+\sigma}}\Bigr) + O\Big(\ell e^{-\frac{3\pi r_\ell }{\ell}}\Big),\\
& =\ell\Big(A_0+\frac{C_1+C_2 + o(1)}{(\ell\ln\ell)^m} \Big).
\end{split}
\end{align}
Therefore we get \eqref{4} with $C_0 = C_1 + C_2 > 0$. \qed

\medskip
Now let us show that the maximum of $F$ over $D_\ell$ cannot be reached on $\partial D_\ell$. Consider first $r_1=(\frac{m}{2\pi}-\delta)\ell\ln\ell$. As $\beta_{12}A_2+A_3 > 0$, for $\ell$ large enough, we have $$B_1 := \beta_{12}A_2+A_3 + K(\beta) > 0.$$
Take any $(r_1, \rho)\in D_\ell,$ there holds, as $\ell \to \infty$
\begin{align*}
F(r_1,\rho)
&\leq \ell A_0+\frac{C\ell}{(\ell\ln\ell)^m}- \big[B_1 + o(1)\big]\ell\frac{e^{-2\pi(\frac{m}{2\pi}-\delta)\ln \ell}}{(\frac{m}{2\pi}-\delta)\ln \ell}+ \ell\widetilde{h}(\rho) + O\Bigl(\frac{\ell}{(\ell\ln\ell)^{m+\sigma}}\Bigr)\\
&\leq \ell A_0 + O\Big(\frac{\ell}{(\ell\ln\ell)^m}\Big) - \frac{C'\ell}{\ell^{m-2\pi\delta}\ln \ell}
\cr
&< \ell A_0.
\end{align*}
Seeing \eqref{4}, clearly the maximum value cannot be realized with $r= r_1 = (\frac{m}{2\pi}-\delta)\ell\ln\ell$.

\medskip
Take now $(r_2, \rho)\in D_\ell$ with $r_2 := M\ell\ln\ell$. We get, for $\ell\to \infty$
\begin{align*}
F(r_2,\rho)
&\leq \ell A_0+ \ell \frac{(a_1\alpha^2+a_2\gamma^2)A_1}{(M\ell\ln \ell)^m} + \widetilde{h}(\rho_\ell) + O\Bigl(\frac{\ell}{(\ell\ln\ell)^{m+\sigma}}\Bigr)
\cr
& \leq \ell A_0+ \ell \frac{(a_1\alpha^2+a_2\gamma^2)A_1}{(M\ell\ln \ell)^m} + \ell\frac{C_2 + o(1)}{(\ell\ln\ell)^m}.
\end{align*}
Let $M$ be large enough satisfying $(a_1\alpha^2+a_2\gamma^2)A_1 < C_1M^m$. As $C_0 = C_1 + C_2$, the above estimate means that the maximum value of $F$ over $D_\ell$ cannot be realized if $r = M\ell\ln\ell$.

\medskip
By similar consideration, we can check that with suitable choice of $\delta$ and $M$, the maximum of $F$ over $D_\ell$ cannot be reached if $\rho_\ell = (\frac{m}{2\pi}-\delta)\ell\ln\ell$ or $\rho_\ell = M\ell\ln\ell$.

\medskip
Finally, we can fix $\delta > 0$ small enough and $M > 0$ large enough such that for $\ell$ large, there is an interior maximum point for $F$ on $D_\ell$. Hence a critical point of $F$ exists, so a solution to the system exists.

\medskip
For case $(i)$ of $(H_m)$, for example, let $m_1<m_2$ and $\,a_1>0$. As now $\frac{a_1\alpha^2}{r^{m_1}}$ is the main order term for the energy expansion, we use it instead of $\frac{a_1\alpha^2+a_2\gamma^2}{r^m}$, similar lower bound estimate as in Lemma \ref{newlem1} holds true, and the same conclusion can be derived. We omit the details.

\medskip
The last step is prove that $u_\ell=U_r+\varphi_\ell,\,v_\ell=V_r+\psi_\ell,\,w_\ell=W_\rho+\xi_\ell$ are positive. First, the negative part ${(u_\ell)}_ -, {(v_\ell)}_-, {(w_\ell)}_-$ tend to zero in $H$ as $\ell\rightarrow\infty$. By regularity theory, we have, as in \cite{PW}, $\|(\varphi_\ell,\psi_\ell,\xi_\ell)\|_\infty$ tends to zero as $\ell\rightarrow\infty$.
We see from $\langle I'(u_\ell,v_\ell,w_\ell),({(u_\ell)}_-,0,0)\rangle=0$ that
\begin{align*}
\|(u_\ell)_-\|_{P_1}^2\leq o(1)\|(u_\ell)_-\|_{P_1}^2, \quad \mbox{as $\ell \to \infty$}.
\end{align*}
Hence $u_\ell\geq 0$. Similarly, there hold $v_\ell \geq 0, w_\ell \geq 0$. By strong maximum principle, we conclude that $(U_r+\varphi_\ell,V_r+\psi_\ell,W_\rho+\xi_\ell)$ is a positive solution of \eqref{eq1}. \qed

\subsection{Sketch proof of Theorem \ref{th1.4}} Since the approach is very similar to that for Theorem \ref{th1.2}, we omit the details and just explain the main difference.

\medskip
The main difference comes from the energy expansion.
Remark that $$(\overline U_{S_k}, \overline V_{S_k}) := (-1) ^k(U_{S_k}, V_{S_k})$$ are always solutions to \eqref{eqs1.3} and $\overline W_{T^k} := (-1)^kW_{T^k}$ still satisfies the equation
\begin{align*}
-\Delta W + W= \mu_3 W^3.
\end{align*}
For the expansion of $I(\overline{U}_r+\overline{\varphi},\overline{V}_r+\overline{\psi},\overline{W}_\rho+\overline{\xi})$, essentially $I_3$ and $I_5$ will have different form comparing to the proof of Proposition \ref{A.1}. Notice also that the number of peaks is now $2\ell$ instead of $\ell$. For example, here we have
\begin{align*}
I_5 & = -\frac{\mu_3}{4}\int_{\R^3}\bar{W}^4_\rho-\sum\limits_{k=1}^{2\ell}W_{T^k}^4-2\sum\limits_{i\neq k}^{2\ell}(-1)^{k+i}W^3_{T^{k}}W_{T^i}\cr
& = -\frac{\mu_3\ell}{2}\int_{\Omega_1}W_{T^1}(-1)^{k+1}\sum\limits_{k=2}^{2\ell}W_{T^k}
=\frac{\mu_3\ell}{2}\int_{\Omega_1}W_{T^1}(-1)^{k}\sum\limits_{k=2}^{2\ell}W_{T^k}\cr
& = A_4\frac{\ell^2}{\rho}e^{-\frac{2\pi\rho}{\ell}}+O(\ell e^{-\frac{3\pi\rho}{\ell}}).
\end{align*}
The search of a critical point with the form $(\overline{U}_r+\overline{\varphi},\overline{V}_r+\overline{\psi},\overline{W}_\rho+\overline{\xi})$ will be reduced to find a critical point of the following function in the interior of $D_{\ell}$:
\begin{align*}
\bar{F}(r,\rho) & = \ell A_0+\Bigl(\frac{a_1\alpha^2}{r^{m_1}}+\frac{a_2\gamma^2}{r^{m_2}}+\frac{a_3}{\mu_3 \rho^{m_3}}\Bigr)\ell A_1+(\beta_{12}A_2+A_3)e^{-\frac{2\pi r}{\ell}}\frac{\ell^2}{r}\cr
&\quad +A_4e^{-\frac{2\pi \rho}{\ell}}\frac{\ell^2}{\rho}+O\Big(\frac{1}{r^{\min\{m_1,m_2\}+\sigma}}+\frac{1}{\rho^{m_3+\sigma}}\Big)
\end{align*}
where $A_i$ are the constants given in Proposition \ref{A.1}. Remark that the coefficients of main terms change the sign, we will consider the minimum of $\bar F$ over $D_\ell$. \qed

\end{document}